\documentclass[12pt,a4paper]{amsart}
\usepackage{mathptmx}
\usepackage{mathrsfs}
\usepackage{verbatim}
\usepackage{url}
\usepackage[all]{xy}
\usepackage{xcolor}
\usepackage{amsmath,amsthm}
\usepackage{amssymb}
\usepackage{enumerate}
\usepackage{graphicx}
\usepackage{epstopdf}

\newtheorem{thm}{Theorem}[section]
\newtheorem{cor}[thm]{Corollary}
\newtheorem{lem}[thm]{Lemma}
\newtheorem{prop}[thm]{Proposition}

\theoremstyle{definition}
\newtheorem{defin}[thm]{Definition}
\newtheorem{rem}[thm]{Remark}
\newtheorem{exa}[thm]{Example}

\numberwithin{equation}{section}

\begin{document}

\baselineskip=17pt

\title[Periodic points for amenable group actions]{Periodic points for amenable group actions on
uniquely arcwise connected continua}

\author{Enhui Shi\ \ \&\  \ Xiangdong Ye}

\address[E.H. Shi]{School of Mathematical Sciences, Soochow University, Suzhou 215006, P. R. China}
\email{ehshi@suda.edu.cn}

\address[X. Ye]
{Wu Wen-Tsun Key Laboratory of Mathematics, USTC, Chinese Academy of Sciences and Department of Mathematics, University of Science and
Technology of China, Hefei, Anhui 230026, China}
\email{yexd@ustc.edu.cn}

\begin{abstract}
 We show that any action of a countable amenable group on a uniquely arcwise connected continuum has a periodic point of
order $\leq 2$.
\end{abstract}
\maketitle

\section{Introduction}

 In this section we will give some basic notions we shall use in the paper. Moreover, we will state the backgrounds and the main theorem.

\subsection{Basic notions}
 Let $X$ be a topological space and let $G$ be a group.
A homomorphism $\phi $ of $G$ into the semigroup of all continuous self-mappings of $X$ is called an {\it
action} of $G$ on $X$. (Note that each $\phi (g)$ is a homeomorphism of $X$ with a continuous inverse $\phi (g^{-1})$.)
For brevity, we usually use $gx$ or $g(x)$ instead of $\phi(g)(x)$ for $g\in G$ and $x\in X$. The {\it orbit}
of $x\in X$ under the action of $G$ is the set $Gx\equiv\{gx:g\in
G\}$. If $Gx$ is finite then $x$ is called a {\it periodic point} of $\phi$
and the cardinality $n$ of $Gx$ is called the {\it order} of $x$; we also say that $x$ is an {\it $n$-periodic point}
 (resp. a {\it fixed point} if $n=1$,  i.e. if $gx=x$ for all $g\in G$).  A subset $Y$ of $X$ is called {\it
$G$-invariant}, if $g(Y)\subset Y$ for all $g\in G$. A Borel measure
$\mu$ on $X$ is called {\it $G$-invariant} if $\mu(g(A))=\mu(A)$ for
every Borel set $A$ in $X$ and every $g\in G$.

{\it Amenability} was first introduced by von Neumann. Recall that a group $G$ is called an {\it amenable group} if there is a
sequence of finite sets $F_i$ ($i=1, 2, 3,\ \dots$) with $\cup_{i=1}^\infty F_i=G$, such that
$\lim\limits_{i\to\infty}\frac{|gF_i\bigtriangleup F_i|}{|F_i|}=0$
for every $g\in G$, where $|F_i|$ is the number of elements in
$F_i$; the set $F_i$ is called a {\it F{\o}lner set}. It is well
known that solvable groups and finite groups are amenable; every
subgroup of an amenable group is amenable. It is also known that any
group containing a free noncommutative subgroup is not amenable. An
important characterization of countable amenable group is that $G$
is amenable if and only if every action of $G$ on a compact metric
space $X$ has a $G$-invariant Borel probability measure on $X$. One
may consult \cite{Pa} for a systematic introduction to amenability.

By a {\it continuum}, we mean a connected compact metric space. A
continuum is {\it nondegenerate} if it is not a single point.  An
{\it arc} is a continuum which is homeomorphic to the closed
interval $[0, 1]$. A continuum $X$ is {\it uniquely arcwise
connected} if for any two points $x\not=y\in X$ there is a unique
arc $[x, y]$ in $X$, which connects $x$ and $y$. A {\it dendrite} is
a locally connected, uniquely arcwise connected continuum. A {\it
tree} is a dendrite which is the union of finitely many arcs.
Clearly, the class of uniquely arcwise connected continua is
strictly larger than that of dendrites. For example, the Warsaw
circle is uniquely arcwise connected but not locally connected.

We provide in the following an example of a uniquely arcwise
connected continuum contained in the plane, which will be repeatedly
mentioned throughout the paper.

\begin{exa}[see Fig.1]
Let $\mathbb R^2$ be the Euclidean  plane.  For each positive integer $n$, let
$I_n$, $I_{-n}$, $J_n$, $J_{-n}$ be the segments between
$(\frac{n-1}{n}, 0)$ and $(\frac{n}{n+1}, 1)$, between
$(-\frac{n-1}{n}, 0)$ and $(-\frac{n}{n+1}, 1)$, between
$(\frac{n}{n+1}, 0)$ and $(\frac{n}{n+1}, 1)$, between
$(-\frac{n}{n+1}, 0)$ and $(-\frac{n}{n+1}, 1)$, respectively. Let
$S^{-}=\cup_{n=1}^{\infty}(I_{-n}\cup J_{-n})$ and
$S^{+}=\cup_{n=1}^{\infty}(I_{n}\cup J_{n})$. Let $L$, $B$, $R$, $M$
be the segments between $(-1,-1)$ and $(-1, 1)$, between
$(-1,-1)$ and $(1, -1)$, between $(1, -1)$ and $(1, 1)$, between
$(0, -1)$ and $(0, 0)$, respectively. Let $X=L\cup B\cup R\cup
M\cup S^- \cup S^+$. Then $X$ is a uniquely arcwise connected
continuum which is not locally connected.
\end{exa}

\begin{figure}[htbp]
\centering
\includegraphics[scale=0.3]{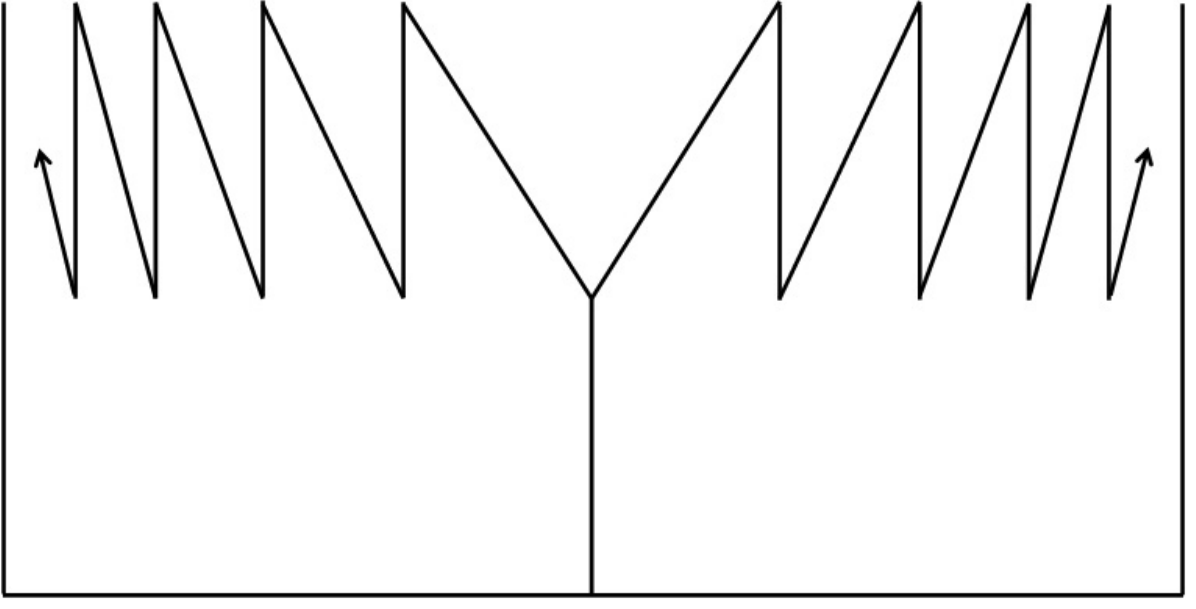}\ \ \ \ \ \ \ \ \ \ \ \ \ \ \ \
\centerline{Fig.1}
\end{figure}

\subsection{Backgrounds and the main theorem}
For an action of a group $G$ on a topological space $X$, an
interesting question is whether there exists a fixed point or a
periodic point of the action. The answer to this question certainly depends on
the topology of $X$ and involves the algebraic structure of $G$.

 In 1975, Mohler proved in \cite{Moh} that every
homeomorphism (i.e., $\mathbb Z$-action) on a uniquely arcwise
connected continuum has a fixed point, which answered a question
proposed by Bing (see \cite{Bi}). In 2009, this result is
generalized to nilpotent group actions by Shi and Sun  (see
\cite{SS}). In 2010, Shi and Zhou further showed that every solvable
group action on such continua has a periodic point of order $\leq 2$ (see \cite{SZ}). 1n 2016, Shi and Ye proved that
every countable amenable group action on a dendrite has a periodic point of order $\leq 2$  (see \cite{SY}). One may
consult \cite{Ho, Ma1, Ma2} for some interesting discussions about
fixed point theory for mappings on uniquely arcwise connected
continua. We also remark that a continuous map on a uniquely arcwise
connected continuum may have no fixed points (see \cite{Yo}).

 We get the following theorem in this paper, which generalizes all
 the corresponding results stated above. (Noting that the integer group is amenable, we see by the following theorem that
 any integer group action on a uniquely arcwise connected continuum preserves either a point or an arc, which implies the existence of fixed points.)

\begin{thm}\label{main-th}
Any action of a countable amenable group on a uniquely arcwise connected continuum has a periodic point of
order $\leq 2$.
\end{thm}

We should note that the set of end points and the set of branch points of a uniquely arcwise connected continuum $X$ are each homeomorphism invariants.
So, if the set of end points of $X$ is finite, or if the set of branch points of $X$ is finite but non-empty, then any group action on $X$ admits periodic points,
which easily implies the existence of a periodic point of order $\leq 2$. However, a uniquely arcwise connected continuum may have infinitely many branch points
and even uncountably many end points. This complicates the problem. We also remark that if the acting group $G$ is the free group $\mathbb Z*\mathbb Z$ (not amenable),
then there does exist counterexamples of $G$ actions on some dendrites without periodic points (see \cite[Theorem 6.1]{SY}); if $G$ is the solvable group
$(\mathbb Z/2\mathbb Z)\ltimes\mathbb Z$, then $G$ can act on the closed interval $[0,1]$ with a periodic point of order $2$ and with no fixed points (see \cite[Remark 1.3]{SZ}).

The proofs of the main theorems in \cite{SS} and \cite{SZ} heavily rely on the existence of a decreasing sequence
of commutator subgroups ending at the identity, which reduces the proof to the case of abelian group actions; 
however, such subgroup sequences do not exist in amenable groups in general.
So, we have to develop some new techniques to overcome this difficulty.

 In Section 2, we introduce some  notions and results concerning
the structure and mapping properties of uniquely arcwise connected
continua. In Section 3, we construct a convex metric on a special
class of arcwise connected subsets of uniquely arcwise connected
continua and study their completions with respect to this metric.
Specially, we establish a connection between the group actions on
dendrites and that on uniquely arcwise connected continua. Based on
the connection established in Section 3 and the main theorem in
\cite{SY}, we prove Theorem 1.2 in Section 4.

\section{Preliminaries}
{In this section we will introduce several notions which are important for the
study of our question, including the convex hulls, dendrites, rays and lines in a
compact metric space, and the quasi-retractions.}

\subsection{Convex hulls}

Let $X$ be a uniquely arcwise connected continuum.  If $S$ is a
subset of $X$, we denote by $[S]$ the intersection of all arcwise
connected subsets containing $S$, and call it the {\it convex hull}
of $S$ in $X$. Clearly, $[S]$ is the minimal one among all the
arcwise connected subsets which contain $S$. We remark here that
$[S]$ need not be compact in general. Denote by $[a, b]$ the  unique
arc in $X$ connecting $a$ and $b$ and by $[a, b), (a, b]$, and
$(a, b)$ the sets $[a,b]-\{b\}$, $[a,b]-\{a\}$, and $[a, b]-\{a,
b\}$, respectively.

The following lemma is clear.

\begin{lem}
If $S$ is a finite set in a uniquely arcwise connected continuum
$X$, then $[S]$ is a tree.
\end{lem}

\begin{exa} In Example 1.1, $[S^-]=S^-$ is not compact; If $S$ is
the finite set consisting of points $(-1, 1)$, $(1, 1)$, and
$(0,0)$, then $[S]=L\cup B\cup R\cup M$, which is a tree.

\end{exa}

\subsection{Dendrites} Let $X$ be a dendrite and let $x\in X$. We
use ${\rm ord}(x, X)$ to denote the cardinality of the set of all
components of $X-\{x\}$, which is called the {\it order of $x$ in
$X$}. The point $x$ is a {\it cut point} if ${\rm ord}(x, X)= 2$;
is a {\it branch point} if ${\rm ord}(x, X)\geq 3$; and is an {\it end
point} if ${\rm ord}(x, X)=1$. For a nondegenerate dendrite $X$,
there are at most countably many branch points, there are
uncountably many cut points, and there always exist end points. One
may consult \cite{Na} for more properties of dendrites.

\begin{prop}
Let $f$ be a homeomorphism on a nondegenerate dendrite $X$. Suppose
$e$ is an end point of $X$ such that $f(e)=e$. Then there is $u\in X,
u\not= e$, such that either $f([e, u])\subset [e, u]$ or $f^{-1}([e,
u])\subset [e, u]$. Moreover, there is $u'\in (e,u]$ such that
$[e,u']\cup f([e,u'])\cup f^{-1}([e,u'])\subset [e,u].$
\end{prop}

\begin{proof}
Fix a point $v\not=e\in X$. Since $e$ is an end point and $f(e)=e$,
there is $w\not=e$ such that $[e, w]=[e, v]\cap [e, f(v)]$. Let
$u=f^{-1}(w)$. If $[e, u]\subset [e, w]$, then $f^{-1}([e,
u])\subset f^{-1}([e, w])=[e, u]$; if $[e, w]\subset [e, u]$, then
$f([e, u])=[e, w]\subset [e, u]$.

{ Moreover, if $f([e, u])\subset [e, u]$, we choose $u'\in (e,u]$
with $f^{-1}(u')=u$, and if $f^{-1}([e,u])\subset [e, u]$, we choose $u'\in (e,u]$
such that $f(u')=u$. Then $u'$ is the point we want.}
\end{proof}

The following two corollaries follow immediately from Proposition
2.3.

\begin{cor}
Let $f_1,...,f_n$ be homeomorphisms on a nondegenerate dendrite $X$. Suppose $e$ is an endpoint of $X$
such that $f_i(e)=e$ for all $i=1,...,n$. Then there are $u,
v\not=e\in X$  such that $f_i([e, v])\cup f_i^{-1}([e, v])\subset
[e, u]$ for all $i=1,...,n$.
\end{cor}

\begin{cor}
Let $f$ be a homeomorphism on a nondegenerate dendrite $X$. Suppose
$e$ is an endpoint of $X$ such that $f(e)=e$. Then there is a
sequence $\{u_i\}_{i=1}^\infty$ in $X$  satisfying the following two
conditions simultaneously  (1) $[u_1, e]\supsetneq [u_2, e]\supsetneq[u_3,
e]\supsetneq...$ and $\cap_{i=1}^{\infty}[u_i, e]=\{e\};$ (2) either
$f([u_i, e])\subset [u_i, e]$ for all $i$, or $f^{-1}([u_i,
e])\subset [u_i, e]$ for all $i$.
\end{cor}

Let $X$ and $Y$ be metric spaces and let $f:X\rightarrow Y$ be
continuous. If ${\rm diam}(f^{-1}(f(x)))\leq\epsilon$ for some
$\epsilon
>0$ and for every $x\in X$, then $f$ is called an {\it
$\epsilon$-map}. A continuum $X$ is {\it tree-like} provided that
for every $\epsilon>0$ there is an $\epsilon$-map $f_\epsilon$ from
$X$ onto some tree $Y_\epsilon$.

\begin{thm}[{\cite[Exercise 10.50]{Na}}]
 A locally connected continuum is tree-like if and only if it is a
 dendrite.
\end{thm}

\begin{thm}[{\cite[Theorem 1.1]{SY}}]
Any action of a countable amenable group on a dendrite has a periodic point of
order $\leq 2$.
\end{thm}

\subsection{Rays and lines}
\begin{defin}
Let $X$ be a compact metric space. If
 $\phi:[0, +\infty)\rightarrow X$ is a continuous injection, then
 $\phi$ or its image $R\equiv\phi([0,+\infty))$ is called a {\it ray} in $X$;
  $R$ or $\phi$ is called {\it oscillatory} (resp.
{\it nonoscillatory}) if $\cap_{n=0}^\infty \overline {\phi([n,
+\infty))}$ contains at least two points (resp. only one point). If
 $\psi:(-\infty, +\infty)\rightarrow X$ is a continuous injection, then $\psi$ or its image
 $L\equiv\psi((-\infty,+\infty))$ is called a {\it line} in $X$;
$L$ or $\psi$ is called {\it oscillatory}  if either
$\cap_{n=0}^\infty \overline {\psi((-\infty, -n])}$ or
$\cap_{n=0}^\infty \overline {\psi([n, +\infty))}$ contains at least
two points;  is called {\it bi-sided-oscillatory} if both
$\cap_{n=0}^\infty \overline {\psi((-\infty, -n])}$  and
$\cap_{n=0}^\infty \overline {\psi([n, +\infty))}$ contain at least
two points; is called {\it one-sided-oscillatory} if it is
oscillatory but not bi-sided-oscillatory;  is called ${\it
nonoscillatory}$ if it is not oscillatory.
\end{defin}
 We should note that if $\phi_1$ and $\phi_2$ are two rays
 with $\phi_1([0,+\infty))=\phi_2([0,+\infty))$, then $\phi_1$
 and $\phi_2$ have the same types
of oscillation. The same conclusion is true for lines. One may
consult \cite{MS} for more information about rays (called ``{\it quasi-arc}" in \cite{MS}).

\begin{defin} Let $X$ be a compact metric space. Let  $R$ be a ray
in $X$ and let $L$ be a line in $X$. We say that $L$ is an {\it
extension} of $R$ if there is a continuous injection $\phi:(-\infty,
+\infty)\rightarrow X$ such that $L=\phi((-\infty, +\infty))$ and
$R=\phi([0, +\infty))$.
\end{defin}

The following lemma will be used in the proof of the main result.
\begin{lem}
Let $X$ be a uniquely arcwise connected continuum. Let $R$ be a ray
in $X$ and let $\phi:[0, +\infty)\rightarrow X$ be a continuous
injection such that $R=\phi([0,+\infty))$. If there is an arc $[a,
b]$ in $X$ such that $\phi(0)\in (a, b)$, then $R$ can be extended
to a line $L$ in $X$. {Moreover, there is a maximal line extending $R$.}
\end{lem}
\begin{proof} { The first statement is clear. To show the second one,
let $$\mathcal{F}=\{L: L\ \text{is a line which extends}\ R\}.$$ It is clear that
$\mathcal{F}\not=\emptyset$. $\mathcal{F}$ is a partially ordered set with respect to the inclusion of sets.

Assume that $\{L_\lambda\}_{\lambda\in \Lambda}$ is a totally ordered subset of $\mathcal{F}$.
Set $N=\cup_{\lambda\in \Lambda}L_\lambda.$ We claim that $N$ is a line which extends $R$.

In fact, for any $x\not=y\in N$ define $x<y$ if and only if exactly one of the following items holds: (1) $x,y\in R$
and $[\phi(0),x]\varsubsetneq [\phi(0),y]$; (2) $x,y\not\in R$ and $[x,\phi(0)]\supsetneq [y,\phi(0)]$; (3) $x\not\in R$ and $y\in R$.

Then we can check that $"<"$ is a total ordering on $N$. By the compactness of $X$, for any integer $k>0$,
there exists $a_k\in N\setminus R$ such that
\begin{equation}\label{equ-2-10}
N\subset B([a_k,\phi(0)],\frac{1}{k})\cup R.
\end{equation}
Since $N$ contains no minimal element according to the ordering $<$, we may suppose
$$\ldots<a_3<a_2<a_1<0.$$
Let $\psi: (-\infty,+\infty)\rightarrow N$ be a continuous injection such that
$$\psi([-n,0])=[a_n,\phi(0)],\ \forall n\in\mathbb{N},\ \text{and}\ \psi([0,+\infty))=R.$$
It remains to show $\psi((-\infty,0))=N\setminus R$.

Assume the contrary that there exists
$z\in N\setminus \psi((-\infty,+\infty))$. Then $z<y$ for any $y\in \psi((-\infty,+\infty))$.
Take $z'\in N$ such that $z'<z$ (noting that $N$ contains no minimal element). Then we have $d(z',[z,\phi(0)])>0$.
This contradicts (\ref{equ-2-10}). Thus, the claim is proved.

Applying Zorn's Lemma, we have proved the existence of the maximal line.}
\end{proof}

\begin{exa} In Example 1.1, $S^-$ and $S^+$ are one-sided
oscillatory rays; $S^-\cup S^+$ is a bi-sided-oscillatory line;
$L-{(-1, 1)}$ is a nonoscillatory ray; $L-\{(-1, 1), (-1, -1)\}$
is a nonoscillatory line; the line $S^-\cup S^+$ is an extension of
the ray $S^-$.
\end{exa}

\subsection{Quasi-retractions}

Let $X$ be a uniquely arcwise connected continuum. Let $Y$ be either
a tree, or an oscillatory ray, or a bi-sided oscillatory line
contained in $X$. Then, by the uniquely arcwise connectivity, for
every $x\in X$, there is a unique $y\in Y$ such that $[x, y]\cap
Y=\{y\}$; we denote $y=r_Y(x)$, and call the map $r_Y:X\rightarrow
Y, x\mapsto r_Y(x)$ the {\it quasi-retraction} from $X$ onto $Y$. We
should note that $r_Y$ is not continuous in general. The idea of quasi-retraction
comes from \cite{Moh}.

Here we remark that the quasi-retraction $r_Y$ cannot be defined
for any arcwise connected subset $Y$ of $X$, unless $Y$ satisfies the
the requirement that for every $x\in X$, there is a unique $y\in Y$ with $[x, y]\cap
Y=\{y\}$; this is why we need to assume $Y$ being of some special form as above.

\begin{lem}
Let $X$ be a uniquely arcwise connected continuum. Let $Y$ be either
a tree, or an oscillatory ray, or a bi-sided oscillatory line
contained in $X$. If $Z$ is an arcwise connected subset of $Y$, then
$r_Y^{-1}(Z)$ is an arcwise connected Borel measurable subset of
$X$.
\end{lem}
\begin{proof}
From  the definition of $r_Y$, we see that if $x\in r_Y^{-1}(Z)$,
then $[x, r_Y(x)]\subset r_Y^{-1}(Z)$, that is every point in
$r_Y^{-1}(Z)$ is connected to a point in $Z$ by an arc in
$r_Y^{-1}(Z)$. Since $Z\subset r_Y^{-1}(Z)$ and $Z$ is arcwise
connected, we know $r_Y^{-1}(Z)$ is arcwise connected. For the
measurability of $r_Y^{-1}(Z)$, one may consult \cite{Moh}.
\end{proof}

\begin{exa}
In Example 1.1, if $Y$ is the tree $L\cup B\cup R\cup M$, then
$r_Y^{-1}((0,0))=S^-\cup S^+$; if $Y$ is the line $S^-\cup S^+$,
then $r_Y^{-1}((0,0))=L\cup B\cup R\cup M$; if $Y$ is the ray
$S^-$, then $r_Y^{-1}((0,0))=L\cup B\cup R\cup M\cup S^+$.
\end{exa}

\section{Induced actions on dendrites}
{In this section we will introduce the notions of the convex metrics and their completions; and
the induced actions.}
\subsection{Convex metrics and their completions}
Let $X$ be a uniquely arcwise connected space (need not be compact).
A metric $d$ on $X$ is {\it convex}, if for any $u, v, x, y\in X$
with $[u, v]\subset [x, y]$, we have $d(u, v)\leq d(x, y)$. Suppose
$T_1\subsetneq T_2\subsetneq T_3\subsetneq...$ is a strictly
increasing sequence of trees contained in $X$. Let
$T=\cup_{i=1}^\infty T_i$. Then $T$ is an arcwise connected subset
of $X$. Clearly, $T$ is also  the union of infinitely many arcs
$I_i$  $(i=1, 2, 3,...)$ with $I_i\cap I_j$ being a point or empty
for any $i\not=j$. Without loss of generality, we may suppose that
$$T_n=\cup_{i=1}^nI_i$$
 for each $n\in \mathbb{N}$. Fix a homeomorphism
$h_i:I_i\rightarrow [0, 1]$ for each $i$. If $[a, b]\subset [0, 1]$,
we denote by $l([a, b])$ the length of the interval $[a, b]$ under
the Euclidean metric on $[0 ,1]$, i.e., $l([a, b])=|a-b|$. For $x,
y\in T$, define
\begin{equation}
d(x, y)=\sum_{i=1}^\infty \frac{1}{2^i}l(h_i([x, y]\cap I_i)).
\end{equation}
It is direct to check that $d$ is a convex metric on $T$. Let
$\widetilde T$ be the completion of $T$ with respect to the metric
$d$. We still use $d$ to denote the naturally induced metric on
$\widetilde T$.

\begin{prop}
$(\widetilde T, d)$ is a dendrite.
\end{prop}
\begin{proof}
For every
$\epsilon>0$, there is $n\in \mathbb{N}$ such that
\begin{equation}
\sum_{i=n+1}^\infty\frac{1}{2^i}<\frac{\epsilon}{5}.
\end{equation} Noting that
for every $\epsilon'>0$, by the convexity of $d$, we always have
$$d(r_{T_n}(x), r_{T_n}(y))\leq d(x, y)<\epsilon',$$
whenever $x, y\in T$ with $d(x, y)<\epsilon'$. This shows that
$r_{T_n}: (T, d)\rightarrow (T_n, d)$ is uniformly continuous. So,
$r_{T_n}$ can be extended to a continuous map
\begin{equation}
r_\epsilon:(\widetilde T, d)\rightarrow (T_n, d).
\end{equation}
We claim that $r_\epsilon$ is an $\epsilon$-map. Otherwise,
there are $x, y\in\widetilde T$ with $d(x, y)>\epsilon$ and
$r_\epsilon(x)=r_\epsilon(y)$. Then by the continuity of
$r_\epsilon$ and the density of $T$ in $\widetilde T$, there are $x',
y'\in T$ such that $d(x, x')<\frac{\epsilon}{5}$, $d(y,
y')<\frac{\epsilon}{5}$, and $d(r_{\epsilon}(x'),
r_{\epsilon}(y'))<\frac{\epsilon}{5}$. So, by (3.2), we have
$$ d(x,
y)\leq d(x, x')+d(x',
r_\epsilon(x'))+d(r_\epsilon(x'),r_\epsilon(y'))+d(r_\epsilon(y'),y')+d(y,y')<\epsilon,
$$ which is a contradiction. By the arbitrariness of $\epsilon$, we
get simultaneously that  $(\widetilde T, d)$ is totally bounded and hence compact; is locally connected since $r_\epsilon$
is monotone by Lemma 2.12 (see \cite[8.4]{Na}); is tree-like. It follows from  Theorem 2.6 that $(\widetilde T, d)$ is a
dendrite.

\end{proof}

\begin{rem}
Though the inclusion $i : (T,d) \rightarrow X$ is a continuous injection,
it is not necessarily an embedding. The topology on $(T,d)$ is the weak topology on $\{T_n\}$. That is, $A\subset T$ is $d$-closed
if $A\cap T_n$ is closed for every $n$.
\end{rem}

\begin{exa}
In Example 1.1, if we let $T_n=B\cup M\cup_{i=1}^n(I_i\cup
I_{-i}\cup J_i\cup J_{-i})$ and let $T=\cup_{n=1}^\infty T_n$, then
the completion of $T$ with respect to the metric $d$ defined above
is homeomorphic to the tree `` H ".
\end{exa}

\subsection{Induced actions}
 Let $X$ be a uniquely arcwise connected continuum. Let
$G$ be a countable group acting on $X$. Suppose $G=\{g_i:
i=1, 2, 3, ...\}$. Take a point $p\in X$. For each positive integer
$n$, let $S_n=\{g_i(p):i=1,...,n\}$ and let $T_n=[S_n]$. Then we get
an increasing sequence of trees:
\begin{equation}
T_1\subset T_2\subset T_3\subset ... .
\end{equation}
 Set $T=\cup_{n=1}^\infty T_n \ (=[Gp])$. Then $T$ is a $G$-invariant uniquely
arcwise connected subset of $X$. We assume that $T$ is not a tree.
Then by deleting some $T_i$'s in (3.4) and renumbering  the
remaining $T_i$'s, we can assume that the sequence in (3.4) is
strictly increasing. It follows from Proposition 3.1 that the
completion $(\widetilde T, d)$ of $T$ with respect to the metric $d$
defined in (3.1) is a dendrite.

\begin{prop}
For each $g\in G$, the homeomorphism $g$ on $(T,d)$ is uniformly continuous with respect to
the metric $d$.
\end{prop}
\begin{proof}
Let $g\in G$. For every $\epsilon>0$, there is $m$ such that
\begin{equation}
\sum_{i=m+1}^\infty\frac{1}{2^i}<\frac{\epsilon}{3}.
\end{equation}
Take a sufficiently large $n$ so that $S_n\supset S_m\cup
g^{-1}S_m$. Then $g(S_n)\cap S_n\supset S_m$ and
\begin{equation}
g(T_n)\cap T_n\supset T_m.
\end{equation}
By the compactness of $T_n$, there is $\delta>0$ such that
\begin{equation}
d(g(x),g(y))<\frac{\epsilon}{3},
\end{equation}
whenever $x, y\in T_n$ with $d(x, y)<\delta$.

For any $u, v\in T$, let $u'=r_{T_n}(u)$ and $v'=r_{T_n}(v)$. Then,
by the convexity of $d$, $d(u', v')<\delta$ whenever $d(u,
v)<\delta$. Then by (3.5) (3.6) and (3.7) we have
$$
d(g(u), g(v))\leq d(g(u), g(u'))+d(g(u'), g(v'))+d(g(v'),
g(v))<\epsilon,
$$
provided that $d(u, v)<\delta$. This completes the proof.

\end{proof}

From Proposition 3.4, we know that  every $g\in G$ can uniquely be extended
to a continuous map $\overline g: (\widetilde T, d)\rightarrow
(\widetilde T, d)$. It follows that such an extension of $g^{-1}$ is an inverse to $\overline{g}$, whence:

\begin{prop}
For each $g\in G$, $\overline g: (\widetilde T, d)\rightarrow
(\widetilde T, d)$ is a homeomorphism.
\end{prop}

From Proposition 3.4 and Proposition 3.5, we obtain an action of $G$
on the dendrite $(\widetilde T, d)$ by homeomorphisms, which is
called the {\it induced action} from the $G$-action on $T$.

\section{Proof of the main theorem}
In this section, we start to prove the main result of the paper, namely Theorem 1.2. Let $X$ be a
uniquely arcwise connected continuum and let $G$ be a countable
amenable group. We want to show that every  $G$-action on $X$ has
a periodic point of order $\leq 2$.

 Fix a point $p\in X$. Let $T=[Gp]$ be the convex hull of its orbit. Then $T$ is
an arcwise connected $G$-invariant subset of $X$. If $T$ is a tree,
then $G$ has a periodic point of order $\leq 2$ in $T$ by Theorem 2.7.
So, we may as well assume that $T$ is not a
tree. Thus by the discussion in Section 3, there is a metric $d$ on
$T$ such that the completion $(\widetilde T, d)$ is a dendrite and
there is an induced $G$-action on $(\widetilde T, d)$ by
homeomorphisms. It follows from Theorem 2.7 that there is a periodic point
$q\in \widetilde T$ of order $\leq 2$. If $q\in T$, then the conclusion of Theorem 1.2 holds,
since $q\in X$.

So, we may assume that $q\in \widetilde T-T$, that is
$q$ is an endpoint of $\widetilde T$. If $q$ is a 2-periodic point of
$G$, then $H\equiv\{g\in G:g(q)=q\}$ is a subgroup of $G$ with index
$2$. Notice that $H$ is also amenable and $q$ is a fixed point of
$H$. In this case, if we can show that $H$ has a fixed point $w\in
X$, then $w$ is a periodic point of $G$ with order $\leq 2$, and
the conclusion of Theorem 1.2 holds. So, to show Theorem 1.2 it remains to prove the
following theorem.

\begin{thm}
If the induced $G$-action on $(\widetilde T, d)$ has a fixed point
$q\in \widetilde T-T$, then $G$ has a fixed point in $X$.
\end{thm}

\begin{proof} We divide the proof into two steps.

\medskip
\noindent{\bf Step 1}: We assume first that $G$ is finitely generated with a generator set
$\{g_1,...,g_n\}$ for some  $n\in \mathbb{N}$.

Fix a point $o\in
T$, then $[o, q)\subset T\subset X$. Let $\phi:[0,
+\infty)\rightarrow X$ be a continuous injection with $[o,
q)=\phi([0, +\infty))$. Then $[o, q)$ becomes a ray.

\noindent{\bf Case 1.} $\phi$ is nonoscillatory. Then there is $z\in
X$ such that $z=\cap_{n=1}^\infty{\overline{\phi([n, +\infty))}}$.
{ Since $q$ is fixed by every $g\in G$, we get from Corollary 2.5 a sequence of
$g$ or $g^{-1}$ invariant decreasing sequence $[u_i, q]$ with $\cap_{i=1}^{\infty}[u_i, q]=\{q\}$. So,
$[u_i, q)$ is a sequence of semi-open intervals in $X$, which is $g$ or $g^{-1}$ invariant.
Thus, $z$ as a limit point of $u_i$ in $X$ is also $g$ invariant.}

\noindent{\bf Case 2.} $\phi$ is oscillatory. By Corollary 2.4,
there are $c_2>c_1>0$ such that, for all $i=1,...,n$,
\begin{equation}
g_i(\phi([c_2,+\infty)))\cup g_i^{-1}(\phi([c_2,+\infty)))\subset
\phi([c_1,+\infty)).
\end{equation}
Let $\prec$ be an ordering on $\phi([0,+\infty))$ defined by
$\phi(t)\prec\phi(s)$ if and only if $t<s$, for any $t, s\in
[0,+\infty)$.

First we assume that $z={\rm sup}_\prec\{\{g(\phi(c_2)):g\in G\}\cap
\phi([c_1,+\infty))\}\prec+\infty$. We claim that $z$ is a fixed point of $G$.
In fact, since $e\in G$, we have $\phi(c_2)\preceq z$. By the definition of $z$, we have $g_i^{-1}(z)\preceq z$
for each $1\le i\le n$.
As each $g_i$ preserves the ordering $\prec|_{\phi([c_2,+\infty))}$, the restriction of $\prec$ to
$\phi([c_2,+\infty))$, we have $z\preceq g_i(z)$, which implies that
for each $1\le i\le n$, $g_i(z)=z$.
That is, $z$ is a fixed point of $G$, thus we get the conclusion. So, we may assume
that
\begin{equation}
{\rm sup}_\prec\{\{g(\phi(c_2)):g\in G\}\cap
\phi([c_1,+\infty))\}=+\infty.
\end{equation}
Consider the set $M=\cup_{g\in G}g(\phi([c_2, +\infty)))$. Then $M$
is arcwise connected, since for any $h_1\not=h_2\in G$ there is some
$c'>0$ such that $h_1(\phi([c_2, +\infty)))\cap h_2(\phi([c_2,
+\infty)))\supset \phi([c', +\infty))$ by Corollary 2.4. By Lemma
2.10, {we can take a maximal line $\psi:(-\infty, +\infty)\rightarrow
M\subset X$ with respect to the inclusion relation of subsets such that $\psi([0,+\infty))=\phi([c_2, +\infty))$.} Set
$L=\psi((-\infty, +\infty))$.

{\bf Subcase 2.1.} $\psi$  is bi-sided-oscillatory in $X$. For each
integer $n$, let $L_n=\psi((n,n+1])$ and let $K_n=\{x\in X:
r_L(x)\in L_n\}$. By Lemma 2.12, each $K_n$ is an arcwise connected
Borel measurable set in $X$. Clearly, these $K_n$ form a partition
of $X$. Since $G$ is amenable, there is a $G$-invariant probability
Borel measure $\mu$ on $X$. Suppose $\mu(K_m)>0$ for some integer
$m$. Since $\psi(m)\in M$, there is some $g'\in G$ such that
$g'(\psi(m))\in \phi([c_2,+\infty))$, which implies
$r_L(g'(K_m))\subset \phi((c_2,+\infty))$. Set
$R=\phi((c_2,+\infty))$. Then
\begin{equation}
\mu(r_L^{-1}(R))\geq\mu(g'(K_m))=\mu(K_m)>0.
\end{equation}
 However, by (4.2), we can
take a sequence $s_i\in G$ such that $s_1(\phi(c_2))\prec
s_2(\phi(c_2))\prec s_3(\phi(c_2))\prec...\in R$ and
$s_i(\phi(c_2))\rightarrow+\infty$ as $i\rightarrow\infty$, with
respect to the ordering $\prec$. Then we have
\begin{equation}
0=\mu(\emptyset)=\mu(\cap_{i=1}^\infty
s_i(r_L^{-1}R))=\lim\limits_{i\to\infty}\mu(s_i(r_L^{-1}R))=
\mu(r_L^{-1}R).
\end{equation}
Since (4.3) and (4.4) are contradict to each other, this subcase
does not occur.

{\bf Subcase 2.2.} $\psi$ is one-sided-oscillatory in $X$. Since
$\phi$ is oscillatory, there must exist a point $z\in X$ such that
$z=\cap_{n=1}^{\infty}\overline{\psi((-\infty, -n])}$. If $z$ is a
fixed point of $G$, then the conclusion holds; otherwise, there is
some $\tilde g\in G$ with $\tilde g(z)\not=z$. {Since $\psi$ is maximal,} there is $r\in (-\infty,
+\infty)$ such that
$$
\psi([r, +\infty))=\psi((-\infty, +\infty))\cap \tilde
g(\psi((-\infty, +\infty))).
$$
Denote $w=\psi(r)\in M$. Take $a\in (z, w)$ with $\tilde g(a)\in
(\tilde g(z), w)$. Let $t\in (-\infty, +\infty)$ be such that
$\psi(t)=a$. Set $L'=L\cup \{z\}$, and set $P_t=\{x\in X: r_{L'}(x)\in [z, a]\}$ and $Q_t=\{x\in
X: r_{L'}(x)\in \psi((t, +\infty))\}$ (See Fig.2.). Then, by Lemma
2.12, $P_t$ and $Q_t$ are arcwise connected and Borel measurable,
and $X=P_t\cup Q_t$ (disjoint union). Since $G$ is amenable, there
is a $G$-invariant Borel probability measure $\mu$ on $X$. Then
$1=\mu(X)=\mu(P_t)+\mu(Q_t)$. Noting that $\tilde g(P_t)\subset
Q_t$, we have
$$
\mu(Q_t)\geq \mu(\tilde g(P_t))=\mu(P_t)>0
$$
provided that $\mu(P_t)>0$. Thus we always have $ \mu(Q_t)>0.$ Since
$a\in M$, there is some $g\in G$ such that $g(a)\in \phi([c_2,
+\infty))$. Then, by an argument similar to that in Subcase 2.1, we
get a contradiction.

Altogether, we finish the proof of Theorem 4.1 under the assumption
that $G$ is finitely generated.

\medskip
{\noindent} {\bf Step 2:} Now, suppose that $G$ is not
finitely generated. For any finite subset $F$ of $G$, let $\langle
F\rangle$ be the subgroup of $G$, which is generated by $F$. Define
$$
X_F=\{x\in X: x\ \mbox{is  a fixed point of}\ \langle F\rangle\}.
$$
Then $X_F$ is a nonempty closed subset of $X$. If $F'$ is another
finite subset of $G$, then $X_F\cap X_{F'}=X_{F\cup
F'}\not=\emptyset.$ Thus the family of compact sets $\{X_F: F\
\mbox{is a finite subset of}\ G\}$ has the finite intersection
property. Hence
$$\cap\ \{X_F: F\ \mbox{is finite in}\ G
\}\not=\emptyset, $$ every point of which is a fixed point of $G$.
Thus we complete the proof of Theorem 4.1.
\end{proof}

\begin{figure}[htbp]
\centering
\includegraphics[scale=0.4]{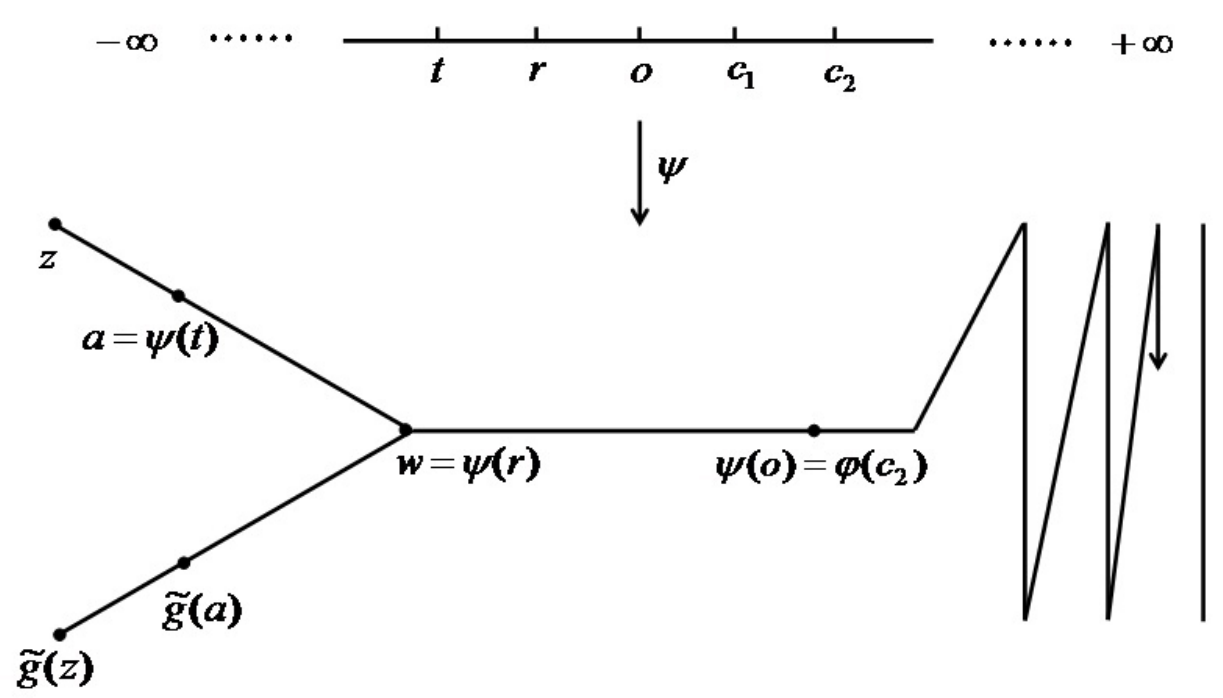}\ \ \ \ \ \ \ \ \ \ \ \ \ \ \ \
\centerline{Fig.2}
\end{figure}

We end this paper with a remark. For any countable group $G$, if $p$ is a periodic point then $[Gp]$ is a tree. It follows that if for some bound $M$ every $G$ action on a tree has a periodic point with period at most $M$ then the same is true of any $G$ action on a uniquely arcwise continuum or else there exists an action of $G$ on such a space with no periodic points. Furthermore, our proof shows that if for some bound $M$ every $G$ action on a dendrite has a periodic point with period at most $M$ then the same is true of any $G$ action on a uniquely arcwise continuum which happens to admit an invariant measure.

\subsection*{Acknowledgements}
We would like to thank
Prof. Hanfeng Li for very helpful suggestions. We are also very thankful to the referee for valuable and thoughtful comments.

The work is supported by NSFC (No. 11771318,
11790274, 11431012).

\end{document}